\documentclass[11pt,a4paper]{article}

\RequirePackage[numbers]{natbib}
\RequirePackage[colorlinks,citecolor=blue,urlcolor=blue]{hyperref}

\usepackage{fullpage,lmodern,bm}
\usepackage[T1]{fontenc}
\usepackage[utf8]{inputenc}                        
\usepackage[Lenny]{fncychap}
\usepackage[usenames,dvipsnames]{xcolor}
\usepackage{amsmath,amssymb,amsthm,thmtools,dsfont,enumerate}
\usepackage{tikz,epigraph,lipsum}
\usetikzlibrary{shapes,arrows}
\usepackage{rotating,lipsum,pifont,cases}
\usepackage{appendix,authblk}
\usepackage{cases}
\usepackage{pdfsync}
\usetikzlibrary{patterns,intersections}
\definecolor{blue0}{RGB}{0,77,153} 
\definecolor{red0}{RGB}{179,0,77} 
\definecolor{green0}{RGB}{134,219,76} 
\definecolor{gray0}{RGB}{84,97,110}

\numberwithin{equation}{section}

\newtheorem{theorem}{Theorem}[section]
\newtheorem{proposition}[theorem]{Proposition}
\newtheorem{lemma}[theorem]{Lemma}
\newtheorem{definition}[theorem]{Definition}
\newtheorem{remark}[theorem]{Remark}

\DeclareMathOperator{\Tr}{Tr}
\DeclareMathOperator{\Ker}{Ker}

\DeclareMathOperator{\diag}{diag}
\DeclareMathOperator{\Conv}{Conv}

\DeclareMathOperator{\supp}{supp}

\newcommand{\calC}{\ensuremath{\mathcal{C}}}
\newcommand{\calD}{\ensuremath{\mathcal{D}}}
\newcommand{\calF}{\ensuremath{\mathcal{F}}}

\newcommand{\bbE}{\ensuremath{\mathbb{E}}}

\newcommand{\bbR}{\ensuremath{\mathbb{R}}}
\newcommand{\bbS}{\ensuremath{\mathbb{S}}}

\numberwithin{equation}{section}
\def\vs#1{\vspace{#1mm}}

\def\be{\begin{align}}
\def\ee{\end{align}}
\def\b*{\begin{eqnarray*}}
\def\e*{\end{eqnarray*}}

\def\be{\begin{eqnarray}}
\def\ee{\end{eqnarray}}
\def\beq{\begin{equation}}
\def\eeq{\end{equation}}
\def\b*{\begin{eqnarray*}}
\def\e*{\end{eqnarray*}}
\def\bi{\begin{itemize}}
\def\ei{\end{itemize}}

\def \1{{\bf 1}}

\def\={\;=\;}

\def\diag#1{\mbox{\rm diag}\left[#1\right]}

 \def\vs#1{\vspace{#1mm}}

\def \D{\mathbb{D}}
\def \E{\mathbb{E}}
\def \F{\mathbb{F}}

\def \M{\mathbb{M}}

\def \P{\mathbb{P}}

\def \R{\mathbb{R}}

\def\Cc{{\cal C}}
\def\Dc{{\cal D}}
\def\Ec{{\cal E}}
\def\Fc{{\cal F}}

\def\Lc{{\cal L}}

\def\Nc{{\cal N}}

\def\Lc{{\cal L}}

\def\D{{\rm D}}

\newcommand{\geometricinterpretationjumps}{%
	\begin{tikzpicture}[thick,scale=0.9, every node/.style={scale=0.9}]
	\node (C) at (0,0) {};
	\node (r) at (-4,0) {};
	\node (l) at (4,0) {};
	\fill[pattern color=lightgray, pattern=north west lines,looseness=2.2] (r) to[out=-70,in =-110]  (C) to [out=70,in=110] (l) arc (0:90:4.0cm) arc (90:180:4.0cm);
	\draw[name path=boundary, thick,black,looseness=2.2]  (r) to[out=-70,in =-110] coordinate[label={[yshift=-.15em]below:(i)}](A) (C) to [out=70,in=110]coordinate[label={[yshift=-.25em]below:(iii)}](B) (l);
	\draw[ultra thin, dashed]    (C)coordinate[label=below  right:(ii)](C)--+(-110:0.5)coordinate(C1)--+(70:0.5)coordinate[label={[xshift=0]right:\textit{\textcolor{blue0}{C}}}](C2);
	\draw[ultra thin, dashed]  (A)--+(0:1.05)coordinate[label={[yshift=-.25em]below:\textit{\textcolor{blue0}{C}}}](A1)--+(180:1.05)coordinate(A2);
	\draw[ultra thin, dashed]  (B)--+(0:1.05)coordinate(B1)--+(180:1.05)coordinate[label={[xshift=0.5em, yshift=0]above:\textit{\textcolor{blue0}{C}}}](B2);
	\draw[->,  ultra thick, dashed , red0]  (A1)--+(90:1)coordinate[label={[yshift=0]\textit{\textcolor{red0}{b}}}](A3);
	\draw[->,  ultra thick, dashed , red0]  (B2)--+(-90:0.5)coordinate[label={[xshift=0,yshift=0]left:\textit{\textcolor{red0}{b}}}](B3);
	\draw[->, ultra thick, dashed , red0]  (B2)--+(90:0.5)coordinate(B3);
	
	\draw[->,  ultra thick, dashed , red0]  (C)--+(-200:0.5)coordinate[label={[yshift=0]\textit{\textcolor{red0}{b}}}](C3);
	\draw[<->, ultra thick, blue0] (A1) -- (A2);
	\draw[<->, ultra thick, blue0] (B1) -- (B2);
	\draw[<->, ultra thick, blue0] (C1) -- (C2);
	
	\node (s0) at (-1.4,1.4) {\textcolor{green0}{$\rho$}};
	\draw[->, ultra thick, dashed, green0] (s0)--+(-160:1.5);
	\draw[->, ultra thick, dashed, green0] (s0)--+(20:1.5);
	\draw[->, ultra thick, dashed, green0] (s0)--+(80:1.5);
	\draw[->, ultra thick, dashed, green0] (s0)--+(-100:1.5);
						
	\node at (-2.2,2.5) {\Large $\calD$}; 
	\end{tikzpicture}	
}%

\title{{Stochastic invariance of closed sets  for  jump-diffusions with non-Lipschitz coefficients}}	
\date{\today}

\author{Eduardo {Abi Jaber} \thanks{abijaber@ceremade.dauphine.fr,  {eduardo.abijaber@axa-im.com. I would like to thank Bruno Bouchard and Camille Illand for very fruitful discussions and insightful comments.}}}

\affil{Universit\'e Paris-Dauphine, PSL Research University, CNRS, UMR [7534], CEREMADE, 75016 Paris, France.}
\affil{AXA Investment Managers,  Multi Asset Client Solutions, Quantitative Research, \break
                        6 place de la Pyramide, 92908 Paris - La Défense, France.}

\begin{document}

	\maketitle

	\begin{abstract}
	We provide necessary and sufficient first order geometric conditions for the stochastic invariance of a closed subset of $\R^d$ with respect to a jump-diffusion under weak regularity assumptions on the coefficients. Our main result extends the recent characterization proved in Abi Jaber, Bouchard and Illand (2016) to jump-diffusions. We also derive an equivalent formulation in the semimartingale framework.\\ \\

		\noindent  {\textit{Keywords:}} Stochastic differential equation, jumps, semimartingale,  stochastic invariance. \\ 
		
		\noindent  {\textit{MSC Classification:}}  93E03,  60H10, 60J75.
		
	\end{abstract}

\section{Introduction}

We consider a weak solution to the following stochastic differential equation with jumps
\begin{equation}\label{diffusionsdeinvariance}
dX_t=b(X_t)dt+\sigma(X_t)dW_t+\int \rho(X_{{t-}},z) \left(\mu(dt,dz) - F(dz)dt \right), \quad X_0=x,
\end{equation}

that {is:} a filtered probability space $(\Omega, \Fc, \F=(\Fc)_{t\geq0},\P)$ satisfying the usual conditions  and supporting a $d$-dimensional Brownian motion $W$, a  Poisson random measure $\mu$ on $\R_+ \times \R^d$ with compensator $dt \otimes F(dz)$, and a $\F$-adapted process $X$ with càdlàg sample paths such that \eqref{diffusionsdeinvariance} holds $\P$-almost surely.\vs2

Throughout this paper, we assume that $b:\R^d \mapsto \R^d$, $\sigma: \R^d \mapsto \M^d$ and    $\rho: \R^d \times \R^d \mapsto \R^d$ are measurable, where $\M^d$ denotes the space of $d \times d$ matrices.  In addition, we assume that 
\begin{align}
b, \sigma \mbox{ and } \int \rho(.,z)^\top H(\rho(.,z))\rho(.,z) F(dz) \mbox{ are continuous for any } H \in \Cc_b(\R^d,\M^d)  \label{conditionscontinuity} \tag{$H_{\Cc}$},
\end{align}
where $\Cc_b(\R^d,\M^d)$ denotes the space of $\M^d$-valued continuous bounded functions on $\R^d$. We also assume that there exist $q,L >0$ such that, for all $x \in \R^d$,
	\begin{align}
	\int_{\{\|\rho(x,z)\| > 1\}}  \|\rho(x,z)\|^q \ln \|\rho(x,z)\| F(dz) &\leq  L (1+ \|x\|^{q}), \label{growthconditionsln} \tag{$H_0$} \\
\|b(x)\|^2 + \|\sigma \sigma^\top(x)\| + \int  \|\rho(x,z)\|^2 F(dz) &\leq  L (1+ \|x\|^2).  \label{growthconditions} \tag{$H_1$} 
\end{align}

Let $\Dc$ denote a closed subset of $\R^d$. Our aim is to characterize the stochastic invariance (a.k.a~viability) of $\Dc$ under weak regularity assumptions, i.e.~find necessary and sufficient conditions on the coefficients such that, for all $x \in \Dc$, there exists a $\Dc$-valued  {weak} solution to \eqref{diffusionsdeinvariance} starting at $x$.  \vs2

Invariance and viability problems have been intensively studied in the literature, first in a deterministic setup \cite{aub91} and later in a  random environment. For the diffusion case, see \cite{ajbi16,da,buc} and the references therein. In the presence of jumps, we refer to \cite{sim,vee,ftt14}. Note that a first order characterization for  a smooth volatility matrix $\sigma$ is given in \cite{ftt14}, where the Stratonovich drift appears (see \cite{da} for the diffusion case). For a second order characterization, we refer to \cite[Propositions 2.13 and 2.15]{vee}. \vs2

Combining the techniques used in \cite{ajbi16, vee},  we derive for the first time in Theorem \ref{theoremjumps} below, a first order geometric characterization of the stochastic invariance with respect to \eqref{diffusionsdeinvariance} when the volatility matrix $\sigma$ can fail to be differentiable. We also provide an equivalent formulation of the stochastic invariance with respect to semimartingales in Theorem \ref{theoremsemi}. This extends \cite{ajbi16} to the jump-diffusion case.  From a practical perspective, this is the first known first order characterization that could be  directly applied to construct affine \cite{dfs,kal} and polynomial processes \cite{cuchkr} on any arbitrary closed sets, since for these processes the volatility matrix can fail to be differentiable (on the boundary of the domain).  \vs2

In fact, in the sequel,  we only make the following assumption on the covariance matrix  
	\begin{equation} \label{eq: extension C}
	 {C:=\sigma \sigma^\top \mbox{ on } \Dc \mbox{ can be extended to a } C^{1,1}_{loc}(\bbR^d,\bbS^d) \mbox{ function},}\tag{$H_2$}
	\end{equation}
	in which $\mathcal{C}^{1,1}_{loc}$ means $\Cc^{1}$ with {a} locally Lipschitz {derivative} and $\bbS^d$ denotes the set of $d \times d$ symmetric matrices. Note that we do not impose the extension of $C$ to be positive semi-definite outside $\Dc$, so that $\sigma$ might only match with its square-root on $\Dc$. Also, it should be clear that the extension {needs} only to be local around $\Dc$.\vs2
	
	From now on we use the same notation $C$ for  $\sigma\sigma^{\top}$ on $\Dc$ and for its extension defined in Assumption \eqref{eq: extension C}. All identities involving random variables have to be considered in the a.s.~sense, the probability space and the probability measure being given by the context. Elements of $\R^{d}$ are viewed as column vectors. We use the standard notation $I_{d}$ to denote the $d\times d$ identity matrix and denote by $\M^{d}$ the collection of $d\times d$ matrices. We say that $A\in \mathbb{S}^{d}$ (resp.~$\mathbb{S}^{d}_{+}$) if it is a symmetric (resp.~and positive semi-definite) element of $\M^{d}$. {Elements of $\R^{d}$ are viewed as column vectors.}
	Given $x=(x^{1},\ldots,x^{d})\in \R^{d}$, $\diag{x}$ denotes the diagonal matrix whose $i$-th diagonal component is $x^{i}$.  If $A$ is a symmetric positive semi-definite matrix, then $A^{\frac12}$ stands for its symmetric square-root. \vs2

The rest of the paper is organized as follows. {Our main results are stated  and proved in Sections \ref{SectionMainJumps}-\ref{SectionMainSemimartingale}}. {In the Appendix, we adapt to our setting some technical results, mainly from \cite{ajbi16}.}

\section{Stochastic invariance for SDEs}\label{SectionMainJumps}
{In order to ease the comparison with \cite{ajbi16}, we first provide in Theorem \ref{theoremjumps} below a characterization of the invariance for stochastic differential equations with jumps. An equivalent formulation in terms of semimartingales is also provided in the next section (see Theorem \ref{theoremsemi} below). We  insist on the fact that the two formulations are equivalent by the representation theorem of semimartingales with characteristics as in \eqref{eq:semimartingale char} below in terms of a Brownian motion and a Poisson random measure (see \cite[Theorem 2.1.2]{JP11}).}\vs2

We start by making precise the definition of stochastic invariance\footnote{{The concept is also often known as viability. We use the term invariance here in order to stay coherent with the affine/polynpmial literature.}} for stochastic differential equations with jumps.

\begin{definition}[\textit{Stochastic invariance}]\label{def: stock inv} A closed subset $\mathcal{D} \subset \mathbb{R}^d$ is said to be \textit{stochastically invariant} with respect to the jump-diffusion \eqref{diffusionsdeinvariance} if, for all $x \in \mathcal{D}$, there exists a weak solution $X$ to \eqref{diffusionsdeinvariance} starting at $X_0=x$ such that $X_t \in \mathcal{D}$ for all $t\geq0$, almost surely. 
\end{definition}

The following theorem provides a first order geometric characterization of the stochastic invariance using  the (first order) normal cone   $\mathcal{N}_{\mathcal{D}}(x)$ at $x$  consisting of all outward pointing vectors, 
\begin{equation*}\label{eq: def first order cone}
\mathcal{N}_{\mathcal{D}}(x) :=\left\{u \in \mathbb{R}^d: \langle u, y-x \rangle \leq o(\|y-x\|), {\forall\; y \in \calD} \right\}.
\end{equation*}

\begin{theorem}\label{theoremjumps}
	Let $\Dc \subset \R^d$ be closed. Under the continuity assumptions \eqref{conditionscontinuity} and \eqref{growthconditionsln}-\eqref{eq: extension C}, the set $\Dc$ is stochastically invariant with respect to the jump-diffusion \eqref{diffusionsdeinvariance} if and only if 
	\begin{subnumcases}{}
	 x + \rho(x,z) \in \Dc, \; \mbox{for } F\mbox{-almost all } z, \label{ourcondsupp}\\
	\int |\langle u, \rho(x,z) \rangle | F(dz) < \infty, \label{ourcondjump}\\
	C(x) u =0, \label{ourcondcov}\\
	\langle  u, b(x) -\int \rho(x,z) F(dz) -\frac{1}{2} \sum_{j=1}^{d} D C^j(x)(CC^+)^j(x)    \rangle \leq 0, \label{ourconddrift}
	\end{subnumcases}
	for all $x \in \mathcal{D}$ and $u \in \mathcal{N}_{\mathcal{D}}(x)$, in  which $DC^j(x)$ denotes the Jacobian of the $j$-th column of $C(x)$ and $(CC^+)^j(x)$ is the $j$-th column of $(CC^{+})(x)$ with $C(x)^+$ defined as the Moore-Penrose pseudoinverse\footnote{The Moore-Penrose pseudoinverse of a $m \times n$ matrix $A$ is the unique $n \times m $ matrix $A^+$ satisfying: $AA^+A=A$, $A^+AA^+=A^+$, $AA^{+}$ and $A^{+}A$ are Hermitian.} of $C(x)$. 
	\end{theorem}

\begin{figure}[h!]
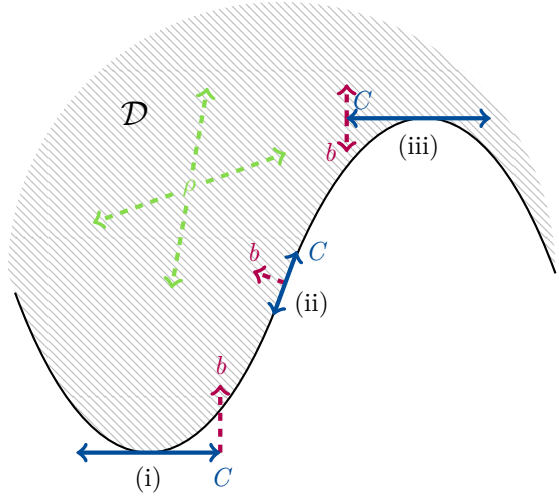

	\centering
	\begin{center}
\geometricinterpretationjumps
	\end{center}
	\rule{35em}{0.5pt}
	\caption[fig:geometricinterpretation]{Interplay between the  geometry/curvature of $\Dc$ and the coefficients $(b,C,\rho)$.}
	\label{fig:geometricinterpretation}
\end{figure}

Before moving to the proof, we start by giving the geometric interpretation of conditions \eqref{ourcondsupp}-\eqref{ourconddrift}, also shown in Figure \ref{fig:geometricinterpretation}. Condition \eqref{ourcondcov} states that at the boundary of the domain, the column of the covariance matrix should be tangential to the boundary, while  \eqref{ourcondsupp} requires from $\Dc$ to capture all the jumps of the process. Moreover, at the boundary, the jumps can have infinite variation only if they are {tangent} to the boundary, by \eqref{ourcondjump}. Finally, it follows from \eqref{ourconddrift} that the compensated drift should be inward pointing. We  notice that the compensated drift  extends the Stratonovich drift (see \cite{da,ftt14}) when the volatility matrix can fail to be differentiable. In fact, if the volatility matrix is smooth, \cite[Proposition 2.4]{ajbi16}  yields  
 
\begin{equation*}
 \langle u, \sum_{j=1}^{d} D \sigma^j(x)\sigma^j(x) \rangle = \langle u, \sum_{j=1}^{d} D C^j(x)(CC^+)^j(x) \rangle, \quad \mbox{for all } x  \in \mathcal{D} \mbox{ and  } u \in \Ker\sigma(x)^\top.
	\end{equation*}
Conversely, the example of the square root process $C(x)=x$ and $\sigma(x)=\sqrt{x}$ on $\Dc:=\R_+$ shows that $\sigma$ may fail to be differentiable at $0$ while $C$ satisfies \eqref{eq: extension C}. \vs2

{The proof of Theorem \ref{theoremjumps} adapts the argument of \cite{ajbi16}  combined with techniques taken from \cite{vee} to handle the jump component.  For the necessity, we use the same conditioning/projection argument together with the small time behavior of double stochastic integrals as in  \cite{ajbi16}. For this we need to inspect the regularity  of $\sigma$, this is the object of Lemma \ref{lemmadistincteigenmatrix} below. For the sufficiency, we show that conditions  \eqref{ourcondsupp}-\eqref{ourconddrift} imply the positive maximum principle for the infinitesimal generator and we conclude by applying \cite[Theorem 4.5.4]{eth}, which is possible by   	Lemma \ref{lemmalcontinuity} below. The latter lemma highlights the role of the growth condition \eqref{growthconditionsln}. In fact,   \eqref{growthconditions} would only yield that $\Lc \phi$ is bounded. This is not enough to apply \cite[Theorem 4.5.4]{eth}.
}
\vs2 

We first recall the following crucial lemma. This is an immediate consequence of the implicit function theorem giving the regularity of the  distinct eigenvalues of $C$ and  their corresponding eigenvectors under  \eqref{eq: extension C}.  We refer to \cite[Lemma 3.1]{ajbi16} for the proof. 
\begin{lemma}\label{lemmadistincteigenmatrix}
	Assume that $C \in \mathcal{C}^{1,1}_{loc}(\mathbb{R}^d,\mathbb{S}^d)$.
	Let $x \in \mathcal{D}$ be such that the spectral decomposition of $C(x)$ is given by
	\begin{equation*}
	C(x)=Q(x)\diag{\lambda_1(x),\ldots ,\lambda_r(x),0,\ldots ,0}Q(x)^\top,
	\end{equation*}
	with $\lambda_1(x) >\lambda_2(x) >\cdots >\lambda_r(x)>0$ and $Q(x)Q(x)^\top=I_d$, $r\le d$.
	
	Then there {exist} an open (bounded) neighborhood $N(x)$ of $x$ and two measurable $\M^{d}$-valued functions on $\bbR^d$, $y \mapsto Q(y):=[q_1(y)\cdots q_d(y)]$ and $y \mapsto \Lambda(y):=\diag{\lambda_1(y),\ldots ,\lambda_d(y)}$ such that
	\begin{enumerate}[{\rm (i)}]
		\item
		$C(y)=Q(y)\Lambda(y)Q(y)^\top$ and $Q(y)Q(y)^\top=I_d$, for all $y \in \bbR^d$,
		\item
		$\lambda_1(y) >\lambda_2(y) >\ldots>\lambda_r(y)>\max\{\lambda_{i}(x), r+1\le i\le d\}\vee 0$, for all $y \in N(x)$,
		\item
		$\bar{\sigma}: y \mapsto \bar{Q}(y)   \bar{\Lambda}(y)^{\frac12}$ is $C^{1,1}(N(x),\mathbb{M}^d)$, in which 
		$\bar{Q}:=[q_1\cdots q_r\; 0\cdots 0]$ and  $\bar{\Lambda}=$ ${\rm diag}[\lambda_1,\ldots,$ $\lambda_r,0,\ldots,0]$.
		
	\end{enumerate}
	
	Moreover, we have: 
	\begin{equation}\label{eqsigmabarc}
	\langle u, \sum_{j=1}^{d} D  \bar{\sigma}^j(x) \bar{\sigma}^j(x) \rangle = \langle u, \sum_{j=1}^{d} D  C^j(x) (CC^+)^j(x)  \rangle, \quad \mbox{for all } u \in \Ker(C(x)).
	\end{equation}
\end{lemma}

We will also need  the continuity of the infinitesimal generator  of \eqref{diffusionsdeinvariance} acting on smooth functions $\phi$ 
	\begin{align}\label{eq:generator}
	\Lc \phi := D\phi b+ \frac 12 \Tr(D^2\phi \sigma \sigma^\top) + \int \left(
	\phi(.+\rho(.,z))-\phi-D \phi \rho(.,z)\right)F(dz),
	\end{align}
	where $D \phi^\top$ (resp.~$D^2 \phi$) is the gradient (resp.~Hessian) of $\phi$.  In the sequel, we denote by $\Cc(\Dc)$ the space of  continuous functions on $\Dc$. {We add the superscript $p$ on $\Cc$ to denote functions with $p$-continuous derivatives for all $p \leq \infty$, and the subscript $c$ (resp.~0) stands for functions with compact support (resp. vanishing at infinity). This is the object of the following lemma  (a similar formulation  in the semimartingale set-up can be found  in  \cite[Lemma A.1]{sp10}).}\vs2

\begin{lemma}\label{lemmalcontinuity}
	Under  \eqref{conditionscontinuity} and \eqref{growthconditionsln},  $\Lc(\Cc^	2_c(\Dc)) \subset \Cc_0(\Dc)$.
\end{lemma}

\begin{proof}
	Let $\phi \in \Cc^2_c(\Dc)$. We extend it to $ \Cc^2_c(\R^d)$. Let $M >0$ be such that $\phi(x)=0$ if $\|x\| > M$ {and fix} $\|x\| > M+1$. Then 
	\begin{align*}
	\Lc \phi(x)= \int \phi(x+\rho(x,z))F(dz)   = \int_{\{\|x+\rho(x,z)\| \leq M\}} \phi(x+\rho(x,z))F(dz). 
	\end{align*}
	On $\{\|x+\rho(x,z)\| \leq M\}$, $1 +M <\|x\| \leq M + \|\rho(x,z)\|$. Hence, \eqref{growthconditionsln} yields
	
	\begin{align*}
	|\Lc \phi(x)| &\leq \| \phi\|_{\infty} \int_{\{\|x+\rho(x,z) \| \leq M\}}    \frac{\|\rho(x,z)\|^q \ln \|\rho(x,z)\|}{(\|x\|-M)^q \ln(\|x\|-M)} F(dz)\\
	&\leq  \|\phi\|_{\infty} L \frac{(1+\|x\|^q)}{(\|x\|-M)^q } \frac 1 {\ln(\|x\|-M)}, 
	\end{align*}
	where $\|.\|_{\infty}$ is the uniform norm, which shows that $\Lc \phi(x) \rightarrow 0 $ when $\|x\| \rightarrow \infty.$
	Moreover, denoting by $\Phi:=\int (\phi(.+\rho(.,z)) - \phi - D\phi \rho(.,z))F(dz)$, we have for all $x,y \in \Dc$
	\begin{align*}
	\Phi(y) &=\int\int_0^1 \int_0^t \rho(y,z)^\top  D^2\phi(y+s\rho(y,z)) \rho(y,z) ds dt F(dz)\\
	&= \int\int_0^1 \int_0^t \rho(y,z)^\top  D^2\phi(x+s\rho(y,z)) \rho(y,z) ds dt F(dz) \\
	&\;\;\;\; + \int\int_0^1 \int_0^t \rho(y,z)^\top  \left(D^2\phi(y+s\rho(y,z))-  D^2\phi(x+s\rho(y,z)) \right) \rho(y,z) ds dt F(dz)\\
	&=:I_1(x,y) + I_2(x,y).
	\end{align*}

{Observe that}	$I_2(x,y)\to 0$ when $y \to x$, since $D^2 \phi$ is uniformly continuous (recall that $\phi$ has compact support). In addition, it follows from \eqref{conditionscontinuity} that $I_1(x,y) \to \Phi(x)$ when $y \to x$, which ends the proof.

\end{proof}

We can now move to the proof of Theorem \ref{theoremjumps}.

\begin{proof}[Proof of Theorem \eqref{theoremjumps}] 
	{{\bf Part a.} We first prove that our conditions are necessary.}  
	Let $X$ denote a  {weak solution} starting at $X_0=x$ such that  $X_t \in \mathcal{D}$ for all $t \geq 0 $. If $x\notin \partial \Dc$, then  $\mathcal{N}_{\mathcal{D}}(x)=\{0\}$ and there is nothing to prove. We therefore assume from now on that $x\in {\partial\Dc}$. Let $0<\eta <1$. Throughout the proof, we fix  $\psi_{\eta}$ a bounded continuous function on $\R^d$  such that $\psi_{\eta}=0$ on $B_{\eta}(x)$ and $\psi_{\eta}\rightarrow\mathds{1}_{\{ \R^d  \setminus  \{0\} \}}$ for $\eta \downarrow 0$, where $B_{\eta}(x)$ is the open ball with center $x$ and radius $\eta$. 
	
	Step 1.  
	We start by proving \eqref{ourcondsupp}. Let  $\epsilon >0$ and $\phi_{\epsilon}: \R^d \mapsto [0,1]$ be $\Cc^2$ such that $\phi_{\epsilon}=0$ on $\Dc \cup B_{\epsilon}(x)$ and $\phi_{\epsilon}=1$ on $\left(\Dc \cup B_{2\epsilon}(x)\right)^c$. $\mathcal{D}$ is stochastically invariant, hence $\phi_{\epsilon}(X_t) =0$,   for all $t \geq 0$. Since $\phi_{\epsilon}$ is twice differentiable and bounded, Itô's formula \cite[Theorem I.4.57]{jas} yields
	\begin{equation*}
	\int_0^t\Lc\phi_{\epsilon}(X_s) +  \int_0^t D\phi_{\epsilon} (X_{s}) \sigma (X_s) dW_s + \left(\phi_{\epsilon}(X_{s-}+\rho(X_{s-},.))-\phi_{\epsilon}(X_{s-}) \right)\ast (\mu - \nu)=0,
	\end{equation*} 
	where $\ast$ denotes the standard notation for stochastic integration with respect to a random measure (see \cite{jas}) and $\nu(dt,dz):=dt F(dz)$. By continuity of $\Lc \phi$ (see Lemma \ref{lemmalcontinuity}), taking the expectation,  dividing by $t$ and letting $t \rightarrow 0$ yield
	\begin{equation}\label{eq:maximumprinciplephiepsilon}
	\Lc \phi_{\epsilon} (x)=0.
	\end{equation}

A change of probability measure with respect to the Doléans-Dade exponential $Z:=\Ec(\psi_{\eta} \ast (\mu - \nu))$, { which is uniformly integrable} {(see \cite[Theorem IV.3]{lep} and the proof of  \cite[ Proposition 2.13]{vee}),} yields
	\begin{equation}\label{eq:itophiepsilon}
	\int_{0}^{t} \widetilde\Lc\phi_{\epsilon}(X_s)ds +  \int_0^t D\phi_{\epsilon} (X_{s}) \sigma (X_s) dW_s + \left(\phi_{\epsilon}(X_{s-}+\rho(X_{s-},.))-\phi_{\epsilon}(X_{s-}) \right)\ast (\mu - \widetilde\nu)=0, 
	\end{equation}
	where 
	\begin{align*}
	&\widetilde b  := b + \int \psi_{\eta}(z) \rho(.,z) F(dz), \quad \widetilde\nu(dt,dz) := dt \widetilde F(dz), \quad \widetilde F(dz) := (1+\psi_{\eta}(z)) F(dz), \\
	&\widetilde\Lc\phi := D \phi \widetilde b + \frac 12 \Tr(D^2\phi C) + \int \left( \phi(.+\rho(.,z)) - \phi - D\phi \rho(.,z) \right) \widetilde F(dz).
	\end{align*}
	
	{By} combining the above with \eqref{eq:generator}, taking the expectation in \eqref{eq:itophiepsilon}, dividing by $t$ and sending $t \rightarrow 0$, {and} invoking once again Lemma \ref{lemmalcontinuity}{,} we get
	\begin{equation*}
	\Lc\phi_{\epsilon}(x) +   \int \phi_{\epsilon}(x+\rho(x,z)) \psi_{\eta}(z) F(dz)=0.
	\end{equation*}
It {then} follows from \eqref{eq:maximumprinciplephiepsilon} that $\int \phi_{\epsilon}(x+\rho(x,z)) \psi_{\eta}(z) F(dz)=0$ for all $\eta \in (0,1)$. Sending $\eta \downarrow 0$ {leads to} $\int \phi_{\epsilon}(x+\rho(x,z)) F(dz)=0$, by monotone convergence (recall that $\phi_{\epsilon}\ge 0$). Hence $$\int \mathds{1}_{\{x + \rho(x+z) \in (\Dc \cup B_{2\epsilon}(x))^c \}} F(dz)=0.$$ {For} $\epsilon \downarrow 0$, \eqref{ourcondsupp} follows from monotone convergence {again}. \vs2

	Step 2. By the proof of \cite[Proposition 3.5]{ajbi16}, it suffices to consider the case where the positive eigenvalues of the covariance matrix $C$ at the fixed point $x \in \Dc$ are all distinct as in Lemma \ref{lemmadistincteigenmatrix}. We can also  restrict the study to $\sigma=C^{\frac 12}$ (see \cite[Remark 2.1]{ajbi16}). We therefore use the notations of Lemma \ref{lemmadistincteigenmatrix}. We proceed as in Step 2 of the proof of \cite[Lemma 3.2]{ajbi16} for the continuous part combined with the proof of \cite[Proposition 2.13]{vee} for the jump part. Fix $u \in \Nc_{\Dc}(x)$ and let  $\phi$ be a smooth function (with compact support in $N(x)$) such that {$\displaystyle\max_{ \mathcal{D}}\phi=\phi(x)$} and $\D\phi(x)=u^\top$.\footnote{Such a function always exists (up to considering an element of the proximal normal cone), see  the discussion preceding \cite[Lemma 3.2]{ajbi16} and Step 1 of the proof of the same Lemma.} Since $\mathcal{D}$ is stochastically invariant, $\phi(X_t) \leq \phi(x)$,   for all $t \geq 0$. Let $w_{\eta}:=(\eta-1)\psi_{\eta}$. By reapplying Step 1, with the test function $\phi$ (resp.~$w_{\eta}$) instead of $\phi_{\epsilon}$ (resp.~$\psi_{\eta}$),   we obtain	\begin{align*}
	0 &\geq \int_{0}^{t} \widetilde\Lc\phi(X_s)ds +  \int_0^t D\phi (X_{s}) \sigma (X_s) dW_s + \widetilde N_t \nonumber \\
	&=  \int_{0}^{t} \widetilde\Lc\phi(X_s)ds +  \int_{0}^{t} (D\phi Q \Lambda^{\frac12}Q^\top)(X_{s})dW_s  + \widetilde N_t{,}
	\end{align*}
	where {$ \widetilde N_s:=\left(\phi(X_{s-}+\rho(X_{s-},.))-\phi(X_{s-}) \right)\ast (\mu - \widetilde\nu)$}  is the pure-jump true martingale part under the new measure ({since $\phi$ is Lipschitz and \eqref{growthconditions} holds}). Let us define the Brownian motion $B=\int_{0}^{\cdot} Q(X_{s})^\top dW_s$, recall that $Q$ is orthogonal, together with $\bar{B}=(B^1,..,B^r,0,...,0)^\top $ and $\bar{B}^{\perp}=(0,...,0,B^{r+1},...,B^d)$. Since $Q {\bar{\Lambda}}^{\frac12}=\bar{Q} {\bar{\Lambda}}^{\frac12}$, the above inequality can be written in the form 
	\begin{align*} 
	0\ge & \int_{0}^{t } \widetilde\Lc\phi(X_{s})ds + \int_{0}^{t} D\phi(X_{s}) \bar \sigma(X_{s})d\bar B_s + \int_{0}^{t} (D\phi Q \Lambda^{\frac12})(X_{s})d\bar B^{\perp}_s + \widetilde N_t.
	\end{align*}

	Let $(\mathcal{F}_s^{\bar{B}})_{s\ge 0}$ be the completed filtration generated by $\bar B$. Since $\bar B,\bar B^{\perp}$ are independent and $\bar B$ has independent increments, conditioning by $\Fc^{\bar B}_t$ yields, by  Lemma \ref{lemmacondexpectationfiltration} in the appendix,
	\begin{align*}
	0\ge & \int_{0}^{t } \mathbb{E}_{\mathcal{F}_s^{\bar{B}}} [\widetilde \Lc \phi(X_{s})]ds + \int_{0}^{t} \mathbb{E}_{\mathcal{F}_s^{\bar{B}}} [D \phi(X_{s}) \bar \sigma(X_{s})]d\bar B_s .
	\end{align*}

	We now apply Lemma \ref{lemmakomlos} of the Appendix to $(D  \phi \bar{\sigma})(X)$ and reapply the same conditioning argument to find a bounded adapted process $\widetilde \eta$ such that 	
	\begin{equation}\label{eq: for applying small time behavior}
	0\ge \int_{0}^{t} \theta_sds + \int_{0}^{t} \left( \alpha + \int_{0}^{s} \beta_rdr +  \int_{0}^{s} \gamma_r dB_r \right)^\top  dB_s,
	\end{equation}
	where 
	\begin{eqnarray*}
		\theta:= \mathbb{E}_{\mathcal{F}_{\cdot}^{\bar{B}}}\left[\widetilde \Lc \phi(X_{\cdot})\right]&,&
		\alpha^{\top}:=  ( D\phi \bar{\sigma})(x) = D\phi (x) Q(x){\Lambda(x)}^{\frac12} \\
		\beta := \mathbb{E}_{\mathcal{F}_\cdot^{\bar{B}}}\left[ \widetilde \eta_{\cdot} \right]&,&
		\gamma :=  \mathbb{E}_{\mathcal{F}_\cdot^{\bar{B}}}\left[D (D \phi  \bar{\sigma})\bar{\sigma}(X_{\cdot})\right].  
	\end{eqnarray*}

	{Step 3.} We now check that we can apply   {Lemma \ref{lemmadoubleintsto}} below.  First note that all the above processes are bounded. This follows from Lemmas \ref{lemmadistincteigenmatrix} and \ref{lemmalcontinuity}, \eqref{growthconditions} and the fact that $\phi$ has compact support. 
	In addition, given $T >0$, the independence of the increments of  $\bar{B}$ implies that $\theta_s= \bbE_{\calF^{\bar{B}}_T}\left[ \widetilde \Lc \phi (X_s) \right]$ for all $s \leq T$.  {From  Lemma \ref{lemmalcontinuity} and since $X$ has almost surely no jumps at $0$, it follows that $\theta$ is a.s.~continuous at $0$.}  {Moreover, since $ D\phi\bar{\sigma}$ is $\mathcal{C}^{1,1}$, $D ( D\phi  \bar{\sigma})\bar{\sigma}$ is Lipschitz  which,  combined with \eqref{eq:momentestimates2}, implies  \eqref{eqgammacontinuity}.}

	{Step 4.} In view of Step 3, we can apply  Lemma \ref{lemmadoubleintsto} to \eqref{eq: for applying small time behavior} to deduce  that $\alpha=0$ and 
	$\theta_0-\frac{1}{2}\Tr(\gamma_0) \leq 0$. The first equation implies that $\alpha^{\top} {\Lambda(x)}^{\frac12} Q^{\top}(x)=u^\top  C(x)=0$, or equivalently \eqref{ourcondcov} since $C(x)$ is symmetric. The second identity combined with $D\phi(x)=u^{\top}$ shows that 
	\begin{align*}
	0&\ge  \widetilde \Lc\phi(x)-\frac12 \Tr\left[ {  \bar \sigma^{\top} D^{2}  \phi \bar \sigma} +( {I_d\otimes u^\top})D \bar{\sigma}\bar{\sigma}\right](x) \\
	 &=  \Lc\phi(x)-\frac12 \Tr\left[ {  \bar \sigma^{\top} D^{2}  \phi \bar \sigma} +( {I_d\otimes u^\top})D \bar{\sigma}\bar{\sigma}\right](x) + (\eta-1)\int \left(\phi(x+\rho(x,z)) - \phi(x)\right) \psi_{\eta}(z) F(dz),
	\end{align*}
	in which $\otimes$ stands for the Kronecker product (see \cite[Definition A.4 and Proposition A.5]{ajbi16}) and $D\bar \sigma$ is the Jacobian matrix of $\bar\sigma$ (see \cite[Definition A.7]{ajbi16}).
	Sending $\eta \downarrow 0$, by monotone convergence, we get 
	\begin{align}\label{eq:equalitytemp}
	0&\ge  \Lc\phi(x)-\frac12 \Tr\left[ {  \bar \sigma^{\top} D^{2}  \phi \bar \sigma} +( {I_d\otimes u^\top})D \bar{\sigma}\bar{\sigma}\right](x) + \int \left(\phi(x)- \phi(x+\rho(x+z))\right) F(dz).
	\end{align}
	
	In particular, since $\phi(x)=\displaystyle \max_{\Dc}\phi $,   \eqref{ourcondsupp} implies that $\int |\phi(x+\rho(x+z)) - \phi(x)|F(dz)=\int (\phi(x) - \phi(x+\rho(x+z)))F(dz) < \infty$. Moreover, the right hand side is equal to 
	\begin{align*}
 -\int D\phi(x)\rho(x,z)F(dz) - \int \int_0^1 \int_0^t \rho(x,z)^\top D^2 \phi(x+s\rho(x,z))\rho(x,z)dsdtF(dz), 
	\end{align*}
	yielding \eqref{ourcondjump} (recall \eqref{growthconditions} and that $\phi$ has compact support). Combining \eqref{eq:equalitytemp},  \eqref{eqsigmabarc}-\eqref{eq:generator} and 
	\begin{align*}
 \Tr\left[(I_d \otimes u^\top)D\bar \sigma \bar \sigma\right](x)= \langle u, \sum_{j=1}^d D\bar \sigma^j(x) \bar \sigma^j(x)  \rangle,
\end{align*}
we finally obtain \eqref{ourconddrift}.\vs2

	{{\bf Part b.} We   now prove that our conditions are sufficient.} It follows from $\eqref{ourcondcov}$ and the proof of \cite[Proposition 4.1]{ajbi16} that  
	\begin{equation*}\label{eq:inequalitytracedrift}
	\Tr(D^2 \phi(x) C(x)) \leq - \langle D \phi(x)^\top, \sum_{j=1}^d DC^j(x)(CC^+)^j(x)\rangle,
	\end{equation*}
	for any smooth function $\phi$ such that $\displaystyle\max_{\Dc} \phi = \phi(x) \geq 0$. Moreover, after noticing that $D\phi(x)^\top \in \Nc_{\Dc}(x)$ (this is immediate from the Taylor expansion of $\phi$ around $x$),  \eqref{ourcondjump} yields
	\begin{align*}
	\int \left(\phi(x+\rho(x,z))-\phi(x)+D\phi(x)\rho(x,z) \right)F(dz) = &\int \left(\phi(x+\rho(x,z))-\phi(x)\right) F(dz) \\
	&+ \int D\phi(x)\rho(x,z) F(dz).
	\end{align*}
	
	 In addition, it follows from \eqref{ourcondsupp}  that  $\phi(x+\rho(x,z)) \leq \phi(x)$ for $F$-almost all $z$. Combining all the above with \eqref{ourconddrift}  we finally get
	\begin{equation*}
	\Lc \phi(x)  \leq	\langle  D \phi(x)^\top, b(x) -\int \rho(x,z) F(dz) -\frac{1}{2} \sum_{j=1}^{d} D C^j(x)(CC^+)^j(x)    \rangle \leq 0.
	\end{equation*}
	Therefore, $\Lc$ satisfies the positive maximum principle.  {In addition, since $\Lc: \Cc_c^{\infty}(\Dc) \mapsto \Cc_0(\Dc)$ (see Lemma \ref{lemmalcontinuity}) and $\Cc_c^{\infty}(\Dc)$ is dense in $\Cc_0(\Dc)$}, by \cite[Theorem 4.5.4]{eth}, there exists a càdlàg $(\Dc \cup \Delta)$-valued solution to the martingale problem for $\Lc$, where $\Delta$ denotes the one point compactification of $\Dc$. $\Delta$ is attained either by jump (killed by a potential) or by explosion. By the discussion preceding \cite[Proposition 3.2]{CFY05}, the process cannot jump to $\Delta$. Moreover, the growth conditions \eqref{growthconditions} ensure that no explosion happens in finite time (see \eqref{eq:momentestimates2}). Hence $\Delta$ is never attained.  {We conclude {by} using \cite[Theorem 2.3]{kur11}.}
\end{proof}

{\section{Equivalent fomulation in the semimartingale framework}\label{SectionMainSemimartingale}}

{In this section, we provide  an equivalent formulation of Theorem \ref{theoremjumps} in the semimartingale set-up which is more adapted to the construction of affine and polynomial jump-diffusions (see Remark \ref{rem:affinejumps} below). We stress once more that, by \cite{elk77,cinja}, \eqref{diffusionsdeinvariance} is a very general formulation, equivalent to  the semimartingale formulation \eqref{eq:canonicaldecomposition} below (see also \cite[Theorem 2.1.2]{JP11}).}\vs2  

Let $X$ denote {a homogeneous diffusion with jumps}  in the sense of  \cite[Definition III.2.18]{jas} on a filtered probability space $(\widetilde \Omega, \widetilde \Fc, \widetilde \F, \widetilde \P)$, i.e.~its semimartingale characteristics $(\widetilde B,\widetilde C,\nu)$ are of the form 
\begin{align}\label{eq:semimartingale char}
\widetilde B_t = \int_0^t \widetilde b(X_s)ds, \quad  \widetilde C_t = \int_0^t \widetilde c(X_s)ds, \quad  \nu(dt,dz) = dt K (X_t, dz),
\end{align}
with respect to a continuous truncation function $h$, {i.e.~$h$ is bounded  and equal to the identity on a neighborhood of $0$}. Here, $\widetilde b : \R^d \mapsto \R^d$, $\widetilde c: \R^d \mapsto \bbS^d_+$, $K$ is a  transition kernel  from $\bbR^d$ into $\bbR^d \setminus\{0\}$ and
\begin{align}
\widetilde b, \widetilde c \mbox{ and } { \int f(z) \|z\|^2 K(.,dz)} \mbox{ are continuous for any bounded continuous function } f  \label{conditionscontinuitytilde} \tag{$\widetilde H_{\Cc}$}.
\end{align}

 The triplet $(\widetilde b, \widetilde c,K)$ is called the differential characteristics of $X$. In addition we assume that there exist $\widetilde q,  \widetilde L>0$ such that 
	\begin{align}
	\int_{\{\|z\| > 1\}}  \|z\|^{\widetilde q} \ln \|z\|  K(x,dz) &\leq  \widetilde L (1+ \|x\|^{\widetilde q}), \label{growthconditionslntilde} \tag{$\widetilde H_0$} \\
\|\widetilde b(x)\|^2 + \| \widetilde c(x)\| +   \int\|z\|^2 K(x,dz) &\leq \widetilde L (1+ \|x\|^2),  \label{growthconditionstilde} \tag{$\widetilde H_1$} 
\end{align}
for all $x \in \R^d$.  {It follows that $X$ is a locally square-integrable semimartingale (see \cite[Definition II.2.27 and Proposition II.2.29]{jas}) and in particular $X$ is a special semimartingale.}  Recall that $\nu$ is the compensated measure of the random jump measure $\mu$ of $X$. By  \cite[Theorem II.2.38]{jas},   $X$ admits the following canonical decomposition
\begin{equation}\label{eq:canonicaldecomposition}
X=X_0 +  B + X^c +  z \ast (\mu- \nu), 
\end{equation}
where $X^c$ is a continuous local martingale with quadratic variation $\langle X^c \rangle_\cdot = \int_0^\cdot  \widetilde c(X_s)ds$ and $B:=\int_0^{\cdot} b(X_s)ds$, where $b:=\widetilde b + \int (z - h(z))K(.,dz)$. Finally, we assume that  
\begin{equation} \label{eq: restriction C}
\mbox{the restriction of } \widetilde c \mbox{ to } \Dc \mbox{ can be extended to a } C^{1,1}_{loc}(\bbR^d,\bbS^d) \mbox{ function},\tag{$\widetilde H_2$}
\end{equation}
and we denote by $C$ this extended function. \vs2

We are now ready to state an equivalent formulation of Theorem \ref{theoremjumps} adapted to \eqref{eq:canonicaldecomposition}. {We start by defining naturally the notion of stochastic invariance with respect to a semimartingale.}  \vs2

 {\begin{definition}[\textit{Stochastic invariance}]\label{def: stock inv 1} A closed subset $\mathcal{D} \subset \mathbb{R}^d$ is said to be \textit{stochastically invariant} with respect to the semimartingale \eqref{eq:semimartingale char} if, for all $x \in \mathcal{D}$, there exists a filtered probability space $(\Omega, \Fc, \F:=(\Fc_t)_{t\geq 0}, \P)$ supporting a semimartingale $X$  with  characteristics  \eqref{eq:semimartingale char} starting at $X_0=x$ and such that $X_t \in \mathcal{D}$ for all $t\geq0$, $\P$-almost surely. 
 \end{definition}}

\begin{theorem}\label{theoremsemi}
Let $\Dc \subset \R^d$ be closed. 	Under the continuity assumptions \eqref{conditionscontinuitytilde} and \eqref{growthconditionslntilde}-\eqref{eq: restriction C}, the set $\Dc$ is \textit{stochastically invariant} with respect to the  semimartingale \eqref{eq:semimartingale char} if and only if 
	\begin{subnumcases}{}
	\supp K(x,dz) \subset \Dc-x, \label{ourcondsupp2}\\
	\int |\langle u, z \rangle | K(x,dz) < \infty, \label{ourcondjump2}\\
	C(x) u =0, \label{ourcondcov2}\\
	\langle  u,   {b(x) -\int z K(x,dz)} -\frac{1}{2} \sum_{j=1}^{d} D C^j(x)(CC^+)^j(x)    \rangle \leq 0, \label{ourconddrift2}
	\end{subnumcases}
	for all $x \in \mathcal{D}$ and $u \in \mathcal{N}_{\mathcal{D}}(x)$.
\end{theorem}

\begin{proof}	
	{Our proof is based on a (standard) representation of \eqref{eq:canonicaldecomposition} in terms of \eqref{diffusionsdeinvariance}. In this proof, we show the correspondence between the characteristics of \eqref{eq:semimartingale char} and the coefficients of \eqref{diffusionsdeinvariance}, and between the assumptions and invariance conditions of the two settings. Then, Theorem \ref{theoremsemi} is deduced from a direct application of Theorem \ref{theoremjumps}.}
	\\
	{{\bf Part a.} More precisely, let us  fix $F$  a $\sigma$-finite and infinite measure with no atom. By  \cite[Lemma 3.4]{cinja} and the discussion preceding \cite[Theorem 3.13]{cinja}, there exists a measurable function $\rho : \R^d \times \R^d \to \R^d \backslash \{0\}$   such that\footnote{There is a lot of freedom for $\rho$, see \cite[Section 4]{cinja} and \cite[Theorem 6 and Corollary 7]{elk77}.}
	\begin{align}
	K(x,B)= \int \mathds{1}_{B}(\rho(x,z)) F(dz), \quad \mbox{ for all Borel sets } B.  \label{eq:identityKFrho}
	\end{align} 
	Let us fix $\rho$ for the rest of the proof, and recall that $b= \widetilde b + \int (z - h(z))K(.,dz)$ and set $ \sigma := {\widetilde c}^{\frac 1 2 }$. We claim that  the assumptions \eqref{conditionscontinuitytilde} and \eqref{growthconditionslntilde}-\eqref{eq: restriction C} imply the assumptions \eqref{conditionscontinuity} and   \eqref{growthconditionsln}-\eqref{eq: extension C}, and  that the conditions  \eqref{ourcondsupp}-\eqref{ourconddrift} are equivalent to the conditions \eqref{ourcondsupp2}-\eqref{ourconddrift2}.
	To see this, recall that  $h$ is bounded  and equal to the identity on a neighborhood of $0$. Hence $z \to (z-h(z))/(\|z\|^2 \mathds{1}_{ \{\|z\| \neq 0\} })$ is bounded and continuous (because it is equal to $0$ on a neighborhood of $0$). It follows, from \eqref{conditionscontinuitytilde} that $b$ and $\sigma$ are continuous. Therefore, combining   the above with  \eqref{eq:identityKFrho} and \eqref{growthconditionslntilde}-\eqref{eq: restriction C} yields \eqref{conditionscontinuity} and \eqref{growthconditionsln}-\eqref{eq: extension C}. Finally,  one easily deduce from \eqref{eq:identityKFrho} that \eqref{ourcondsupp}-\eqref{ourconddrift} are equivalent to \eqref{ourcondsupp2}-\eqref{ourconddrift2} since $\int g(z) K(x,dz)=\int  g(\rho(x,z)) F(dz)$, for any measurable function $g$.\\
	{\bf Part b.} To see that the conditions \eqref{ourcondsupp2}-\eqref{ourconddrift2} of Theorem \ref{theoremsemi} are sufficient, it suffices to apply  Theorem \ref{theoremjumps}, whose assumptions and conditions are satisfied by Part a.~above. Namely,  under the conditions \eqref{ourcondsupp2}-\eqref{ourconddrift2},  \eqref{diffusionsdeinvariance} admits a $\Dc$-valued weak solution, which is also a semimartingale with characteristics \eqref{eq:semimartingale char}. }

{
	{\bf Part c.} We  now prove that the conditions \eqref{ourcondsupp2}-\eqref{ourconddrift2} are necessary. Assume that $\Dc$ is stochastically invariant with respect to \eqref{eq:semimartingale char}. Fix $(\Omega, \Fc, \F:=(\Fc_t)_{t\geq 0}, \P)$ supporting a semimartingale $X$  with  characteristics  \eqref{eq:semimartingale char} starting at $X_0=x \in \Dc$ such that $\P(\{X_t  \in \mathcal{D}, \forall t \geq 0\} )=1$. By \cite[Theorem 2.1.2]{JP11}, there exists a filtered  extension  $(\widetilde \Omega, \widetilde \Fc, \widetilde \F:=(\widetilde \Fc_t)_{t\geq 0}, \widetilde\P)$ supporting a $d$-dimensional Brownian motion $W$ and a Poisson random measure $\mu$  with compensator $dt \otimes F(dz)$ such that $X$ solves 
	\begin{align}\label{eq:diffusiondelta}
	X_t = x + \int_0^t \widetilde b(X_s) ds + \int_0^t \widetilde \sigma_s dW_s  + ( \delta \mathds{1}_{\{ \|\delta\|  \leq 1\}}) \ast (\mu - dt F(dz)) + (\delta \mathds{1}_{\{ \|\delta\| > 1\}}) \ast \mu,
	\end{align}
	where $(\widetilde \sigma,\delta)$ are such that $\widetilde \sigma_t(\widetilde \omega)\widetilde \sigma_t(\widetilde \omega)^{\top} = \widetilde c(X_t(\widetilde \omega))$  and $K(X_t(\widetilde \omega),B ) = \int  \mathds{1}_{ B}(\delta(\widetilde  \omega,t,z)) F(dz)$,  for all Borel sets $B$, for all $t \geq 0 $, for $\widetilde \P$-almost all $\widetilde \omega$. In view of \eqref{eq:identityKFrho}, 
	\begin{align*}
	\int \mathds{1}_{ B} ( \rho(X_{\cdot}{\widetilde \omega},z)) F(dz) = K(X_{\cdot}(\widetilde \omega),B ) = \int  \mathds{1}_{B}( \delta(\widetilde  \omega,\cdot,z)) F(dz),
	\end{align*}
	for all Borel sets $B$, for $\widetilde \P$-almost all $\widetilde \omega$. Hence, $\delta = \rho(X_{\cdot},\cdot)$  $F\otimes \P$ almost everywhere. Similarly, $\widetilde \sigma$ can be taken to be equal to the square root of $\widetilde c$ (see \cite[Remark 2.1]{ajbi16}). Thus, \eqref{eq:diffusiondelta} can be written in the form   \eqref{diffusionsdeinvariance} with  $(b,\sigma:={\widetilde c}^{\frac 12},\rho)$.  Moreover,   $\widetilde \P(\{X_t  \in \mathcal{D}, \forall t \geq 0\}=\P(\{X_t  \in \mathcal{D}, \forall t \geq 0\} )=1$,  by the discussion following \cite[Equation (2.1.26)]{JP11}. In view of Part a.,   Theorem \ref{theoremjumps} implies that  \eqref{ourcondsupp}-\eqref{ourconddrift} should hold, so that   \eqref{ourcondsupp2}-\eqref{ourconddrift2} must be satisfied.} 

\end{proof}

 {
\begin{remark}\label{rem:affinejumps} As already mentioned above, in the presence of jumps, the semimartingale formulation given in Theorem \ref{theoremsemi} is  more adapted to affine and polynomial processes than Theorem \ref{theoremjumps}. In fact, affine (resp.~polynomial) jump-diffusions are characterized by an affine (resp.~polynomial) dependence of their triplet $(\widetilde b,\widetilde c,K)$ (e.g.~\cite[Definition 4.2]{vee}). Inspecting the identity in   \eqref{eq:identityKFrho}, it is not clear how this property translates to $\rho$.   
\end{remark}}

\begin{appendices}

{\section{Technical lemmas}}
{For completeness, we provide in the sequel some technical lemmas with their proofs. They are either standard or minor modifications of already known results.} \vs2

The {generalized It\^{o}'s lemma} derived in \cite[Lemma 3.3]{ajbi16} can easily be extended to account for jumps in the following way. 
\begin{lemma}\label{lemmakomlos}
	Assume that  $\sigma$ is continuous and that there exists a solution $X$  to \eqref{diffusionsdeinvariance}. Let $f \in \mathcal{C}^{1,1}_c (\mathbb{R}^d,\mathbb{R})$. Then, there exists an adapted bounded process $\eta$ such that 	
	\begin{equation*}\label{eqkomlos}
	{f(X_t)= f(x) + \int_{0}^{t} \widetilde \eta_s ds + \int_{0}^{t} (D  f \sigma)(X_s)dW_s + \left(f(X_{s-}+\rho(X_{s-},z))-f(X_{s-}) \right)\ast (\mu - dt F(dz))},
	\end{equation*}
	for all $t\ge 0$,  {with $\widetilde \eta_s=  (D  f b)(X_s) + \eta_s + \int \left(f(X_s+\rho(X_s,z))-f(X_s)-Df(X_s)\rho(X_s,z)\right)F(dz)$.} 
\end{lemma}

{
\begin{proof} Since $f\in \calC^{1,1}$ has a compact support, 	we can find a sequence $(f_n)_{n}$ in $\mathcal{C}^\infty$ with  compact support (uniformly) and a constant $K>0$ such that 
	\begin{enumerate}[(i)]
		\item 
		$\|D^2 f_n\| \leq K$, 
		\item
		$\|f_n-f\|+\|D  f_n-D  f\| \leq \frac{K}{n}$,	
		\end{enumerate}
	for all $n\ge 1$. This is obtained by considering a simple mollification of $f$. Set $\widetilde \mu:=\mu - dt F(dz)$. Since $f_n$ is twice differentiable and bounded,  Itô's formula \cite[Theorem I.4.57]{jas} yields 
	\begin{align*}
	f_n(X_t)&= f_n(x) +  \int_{0}^{t} D  f_n(X_s)  \sigma(X_s)dW_s + \left(f_n(X_{s-}+\rho(X_{s-},z))-f_n(X_{s-}) \right)\ast \widetilde \mu \\
	&\;\;\; + \int_{0}^{t} \left( (D  f_n b)(X_s) + \eta^n_s + \int \left(f_n(X_s+\rho(X_s,z))-f_n(X_s)-Df_n(X_s)\rho(X_s,z)\right)F(dz) \right)ds 
	\end{align*}
	in which $\eta^{n}:=\frac{1}{2}{\Tr[D ^2 f_n\sigma\sigma^{\top}](X)}$.  Since {$\sigma\sigma^{\top}$} is continuous, (i) above implies that $(\eta^{n})_{n}$ is uniformly bounded in $L^{\infty}(dt \times d\mathbb{P})$.  By \cite[Theorem 1.3]{del99}, there exists $(\hat{\eta}^n) \in \Conv( \eta^k, k\geq n )$  such that $\hat{\eta}^n \rightarrow \eta$ $dt\otimes d\mathbb{P}$ almost surely. 
	Let $N_n \geq 0$ and $(\lambda^n_k)_{n \leq k \leq N_n}\subset [0,1]$ be such that 
	$\hat{\eta}^n=\sum_{k=n}^{N_n} \lambda^n_k \eta^k$ and $\sum_{k=n}^{N_n} \lambda^n_k =1$. Set  $\hat{f}_n:=\sum_{k=n}^{N_n} \lambda^n_k f_k$. Then, 	
	\begin{align}
	\hat{f}_n(X_t)&= \hat{f}_n(x) +  \int_0^t \widetilde \eta^n ds + \int_{0}^{t} D  \hat f_n(X_s)  \sigma(X_s)dW_s + \left(\hat f_n(X_{s-} +\rho(X_{s-},z))-\hat f_n(X_{s-}) \right)\ast \widetilde \mu\label{eqitofntilde},
	\end{align}
	in which $\widetilde \eta^n := (D  \hat f_n b)(X_s) + \hat \eta^n_s + \int \left(\hat f_n(X_s+\rho(X_s,z))-\hat f_n(X_s)-D\hat f_n(X_s)\rho(X_s,z)\right)F(dz)$. By   dominated convergence, 
	$
	\int_{0}^{t} \hat{\eta}^n_s  ds$ converges a.s.~to $\int_{0}^{t} \eta_s  ds
	$.  
	Moreover,  (ii) implies that  
	\begin{equation*}
	\|\hat{f}_n(X_t)-f(X_t)\| \leq  \sum_{k=n}^{N_n} \lambda^n_k \|\hat{f}_k(X_t)-f(X_t)\| \leq \sum_{k=n}^{N_n} \lambda^n_k \frac{K}{k}\leq \frac{K}{n},
	\end{equation*} 
	so that   $\hat{f}_n(X_t)$ converges a.s.~to $f(X_t)$. Similarly, 
	\begin{align*}
	   & \int_0^t \widetilde \eta^n_s ds  \rightarrow \int_0^t \widetilde \eta_s ds,  \quad  \int_{0}^{t} D\hat{f}_n(X_s) \sigma(X_s)dW_s  \rightarrow \int_{0}^{t} D f(X_s) \sigma(X_s)dW_s, \\
		&\left(\hat f_n(X_{s-}+\rho(X_{s-},z))-\hat f_n(X_{s-}) \right)\ast \widetilde \mu \rightarrow  \left(f(X_{s-}+\rho(X_{s-},z))-f(X_{s-}) \right)\ast \widetilde \mu,
   \end{align*} 
	in 	$L^2(\Omega,\mathcal{F},\mathbb{P})$ as $n \rightarrow \infty$, and therefore a.s.~after possibly considering a subsequence. It thus remains to send $n\to \infty$ in \eqref{eqitofntilde} to obtain the required result.  
\end{proof}
}

The following adapts \cite[Lemma 3.4]{ajbi16} to our setting.

\begin{lemma}\label{lemmadoubleintsto}
	Let $(W_t)_{t \geq 0}$ denote a standard $d$-dimensional Brownian motion on a filtered probability space $(\Omega,\mathcal{F},(\mathcal{F}_t)_{t \geq 0},\mathbb{P})$. Let $\alpha \in \mathbb{R}^d$ and $(\beta_t)_{t\geq 0}$,  $(\gamma_t)_{t\geq 0}$ and  $(\theta_t)_{t\geq 0}$ be  predictable processes taking values respectively in $\mathbb{R}^d$, $\mathbb{M}^{d }$ and $\mathbb{R}$ and satisfying  
	
	\begin{enumerate}[{\rm (1)}]
		\item
		$\beta$ is bounded,
		\item  $\int_{0}^t \|\gamma_s\|^2 ds < \infty$, for all $ t\geq 0$,
		\item
		there exists  $\eta >0$  such that
		\begin{equation}\label{eqgammacontinuity}
		{\int_0^t \int_0^{s}\E\left[\|\gamma_r - \gamma_0\|^2\right]drds=O(t^{2+\eta})},
		\end{equation}
		\item
		$\theta$ is a.s.~continuous at $0$.
	\end{enumerate}
	Suppose that for all $t \geq 0$
	\begin{equation}\label{eqintegraledouble}
	\int_{0}^{t} \theta_s ds + \int_{0}^{t} \left( \alpha + \int_{0}^{s} \beta_rdr +  \int_{0}^{s} \gamma_r dW_r \right)^\top  dW_s  \leq 0.
	\end{equation}
	
	{Then,  $\alpha=0$, $-\gamma_0 \in \mathbb{S}^d_+$, $\theta_0-\frac{1}{2}\Tr(\gamma_0) \leq 0$.}  \end{lemma}

\begin{proof}	Since $(W^i_t)^2 = 2 \int_0^t W^i_s d W^i_s + t$,   \eqref{eqintegraledouble} reduces to 
	\begin{equation*}
	(\theta_0- \frac{1}{2} \Tr(\gamma_0)) t + \sum_{i=1}^d \alpha^i W^i_t + \sum_{i=1}^d \frac{\gamma^{ii}_0}{2} (W^i_t)^2 + \sum_{1 \leq i \neq j \leq d} \gamma^{ij}_0 \int_0^t W^i_s d W^j_s + R_t \leq 0,
	\end{equation*}
	where
	\begin{eqnarray*}
		R_t &=& \int_{0}^{t} (\theta_s - \theta_0) ds + \int_{0}^{t} \left(  \int_{0}^{s} \beta_rdr \right)^\top  dW_s  +  \int_{0}^{t}\left(\int_{0}^{s} (\gamma_r -\gamma_0) dW_r \right)^\top  dW_s \\
		&=:&  R^1_t + R^2_t + R^3_t.
	\end{eqnarray*}
	In view of    \cite[Lemma 2.1]{buc}, it suffices to show that  $R_t/t \rightarrow 0$ in probability. To see this, first note that $R^1_t = o(t)$  a.s.~since $\theta$ is continuous at $0$. Moreover, \cite[Proposition 3.9]{cst05a} implies that $R^2_t = o(t)$ a.s., as $\beta$ is bounded. Finally,  {it follows from}  \eqref{eqgammacontinuity} {that} $\frac{R^3_t}{t} \rightarrow 0$ in $L^2${, and hence in probability}. We conclude by applying \cite[Lemma 2.1]{buc}.
\end{proof}

We also used the following elementary lemma which extends \cite[Lemma 5.4]{xio} to account for jumps (see also \cite[Corollaries 2 and 3 of Theorem 5.13]{lip01}). 

\begin{lemma}\label{lemmacondexpectationfiltration}
	Let $B,B^{\perp}$ denote two independent Brownian motions and $\mu$ a Poisson random measure on $\R_+ \times \R^d$  with compensator $dt \otimes F $ on a filtered probability space $(\Omega, \Fc, (\Fc_t)_{t \geq 0},\P)$. Let $(\gamma_s)_{s\geq0}$ be an adapted square integrable process and $\xi: \R_+\times \R^d \mapsto \R^d$ be a predictable process such that $ \E\left[\int_0^t\int  \|\xi(s,z)\|^2 F(dz)ds\right] < \infty$, for all $t\geq 0$.  Define the sub-filtration $\mathcal{F}^B_t= \sigma\{B_s, s \leq t\} \subset \mathcal{F}_t$ and denote by $\widetilde \mu=\mu-dtF(dz)$. Then $\mathbb{P}-a.s.$, for all $ t \geq 0$,
	\begin{align*}
	\mathbb{E}_{\mathcal{F}_t^B}\left[ \int_{0}^{t} \gamma_s dB_s \right] &= \int_{0}^{t}  \mathbb{E}_{\mathcal{F}_s^B}\left[ \gamma_s \right]dB_s,  \;\;
	\mathbb{E}_{\mathcal{F}_t^B}\left[ \int_{0}^{t} \gamma_s dB_s^{\perp} \right] = \mathbb{E}_{\mathcal{F}_t^B}\left[ \xi \ast \widetilde \mu  \right] =  0.
	\end{align*}
	Moreover, it holds similarly for any integrable adapted process $\theta$ that
	\begin{align*}
		\mathbb{E}_{\mathcal{F}_t^B}\left[ \int_{0}^{t} \theta_s ds \right] = \int_{0}^{t}  \mathbb{E}_{\mathcal{F}_s^B}\left[ \theta_s \right]ds.  
	\end{align*} 
	 
\end{lemma}

\begin{proof}{
	We sketch the proof for the jump integral. See \cite[Lemma 5.4]{xio}  for the other identities. Let $\xi$ be simple and predicable, \textit{i.e.} $\xi(s,z)= \sum_{i=1}^n \xi_i \mathds{1}_{(t_i,t_{i+1}]}(s) \mathds{1}_{A_i}(z)$, in which $\xi_i$ is bounded and  $\Fc_{t_i}$-measurable, $(t_i)_{1 \leq i \leq n}$ a subdivision of $[0,t]$ and $A_i \subset \R^d$ Borel sets such that $F(A_i) <
	\infty$. We can write 	$\Fc^B_{t}= \Fc^B_{t_i}  \vee \Fc^B_{t_i,t}$ where $\Fc^B_{t_i,t}:= \sigma(B_{s}-B_{t_i}, t_i \leq s \leq t)$. It follows from  \cite[Theorem II.6.3]{ike} that $\mu$ and $B$ are independent and
	\begin{align*}
	\E_{\Fc_t^B} \left[  \xi \ast \widetilde \mu  \right]	&= \sum_{i=1}^n \E_{\Fc^B_{t_i}  \vee \Fc^B_{t_i,t}} \left[  \xi_i \widetilde \mu \left((t_i,t_{i+1}] \times A_i\right)   \right] \\ 
	&= \sum_{i=1}^n \E \left[ \E\left[\xi_i    \widetilde \mu \left((t_i,t_{i+1}] \times A_i\right)  \mid \Fc_{t_i}  \vee \Fc^B_{t_i,t} \right]  \mid \Fc^B_{t_i}  \vee \Fc^B_{t_i,t} \right] \\
	&= \sum_{i=1}^n \E \left[ \xi_i \E\left[     \widetilde \mu \left((t_i,t_{i+1}] \times A_i\right) \mid \Fc_{t_i}  \vee \Fc^B_{t_i,t} \right]  \mid \Fc^B_{t_i}  \vee \Fc^B_{t_i,t} \right] \\
	&=  0.
	\end{align*}	
	For general $\xi$,  the result follows from It\^o's isometry and the fact that simple processes are dense in $L^2(dt \otimes F)$ (see \cite[Lemma 4.1.4]{A04}). }
\end{proof}

For completeness, we recall  well-known moment estimates for  \eqref{diffusionsdeinvariance} under  \eqref{growthconditions}.

\begin{proposition} Let $X$ denote a weak solution of \eqref{diffusionsdeinvariance}  starting at $x$. Under the growth conditions \eqref{growthconditions}, there exists $M^1_{x,L}>0$ such that the following moment estimates hold:
\begin{align}
\E \left[ \sup_{s \leq t}  \|X_s \|^2 \right] &\leq 4\left( \| x\|^2 + Lt(t+8)\right)e^{4Lt(t+8)}, &&\; \mbox{ for all } t\geq0, \label{eq:momentestimates1} \\
\E \left[  \|X_t-X_s \|^2 \right] &\leq   M^1_{x,L}|t-s|, &&\; \mbox{ for all } s, t \leq 1. \label{eq:momentestimates2}
\end{align}
\end{proposition}

	\begin{proof}{
		Set $g_t:= \E \left[ \sup_{s \leq t}  \|X_s \|^2 \right]$. By convexity of  $y \mapsto y^2$, we have $(a+b+c+d)^2=16 ( \frac{a+b+c+d}{4})^2 \leq 4 (a^2+b^2+c^2+d^2)$. Combined with Cauchy–Schwarz and  Burkholder-Davis-Gundy inequalities, we get for all $u \leq t$
		\begin{align*}
		g_u &\leq 4 \left (\|x\|^2 + t \int_0^u \E \left[\|b(X_s)\|^2 \right] ds  +  4 \int_0^u \E \left[\|C(X_s)\|\right] ds  + 4   \int_{[0,u] \times \R^d} \E \left[\|\rho(X_{s-},z)\|^2 \right] F(dz) ds  \right).
		\end{align*}
		The growth conditions \eqref{growthconditions} now yield 
		\begin{align*}
		g_u 
		&\leq 4 \left (\|x\|^2 + Lt( t + 8 ) +  L(t+8)\int_0^u g_s ds \right), \quad \mbox{ for all } u \leq t.
		\end{align*}
		Finally, \eqref{eq:momentestimates1} follows from Grönwall's Lemma. Moreover, for all  $ s, t \leq 1$, by Cauchy-Schwarz inequality, Itô's isometry and \eqref{growthconditions} we have
		\begin{align*}
		\E\left[ \|X_t-X_s \|^2\right] &\leq 3 \left( |t-s|  \int_s^t \E \left[\|b(X_r)\|^2 \right] dr  + \int_s^t \E \left[\|C(X_r)\|+   \int_{\ \R^d} \|\rho(X_{r-},z)\|^2  F(dz)\right] dr\right)\\
		&\leq 3 \left( L|t-s|^2(1+g_1)  + L|t-s|(1+g_1)\right)\\
		&\leq 6L(1+g_1) |t-s|.
		\end{align*}
		Hence, \eqref{eq:momentestimates2} follows from \eqref{eq:momentestimates1}.}
	\end{proof}

\end{appendices}

\bibliographystyle{unsrtnat} 

\end{document}